\documentclass[10pt]{amsart}

\usepackage{amssymb, dsfont} 
\usepackage{fullpage, graphicx} 
\usepackage{hyperref} 
\usepackage{enumitem} 
\usepackage{caption} 
\usepackage{xcolor} 
\usepackage{thmtools} 

\hypersetup{colorlinks=true, urlcolor=purple, citecolor=blue, linkcolor=red, pdfstartview={XYZ null null 0.90}}

\DeclareMathOperator{\nrd}{nrd} 
\DeclareMathOperator{\Emb}{Emb}

\DeclareMathOperator{\Nm}{Nm}

\DeclareMathOperator{\Mat}{Mat}
\DeclareMathOperator{\SL}{SL}
\DeclareMathOperator{\PSL}{PSL}

\DeclareMathOperator{\new}{new}

\DeclareMathOperator{\IS}{IS}
\newcommand{\Z}{\ensuremath{\mathbb{Z}}}
\newcommand{\R}{\ensuremath{\mathbb{R}}}
\newcommand{\Q}{\ensuremath{\mathbb{Q}}}
\newcommand{\C}{\ensuremath{\mathbb{C}}}
\newcommand{\uhp}{\ensuremath{\mathbb{H}}} 
\newcommand{\Ord}{\ensuremath{\mathrm{O}}}
\newcommand{\YO}{\ensuremath{Y_{\Ord}}}
\newcommand{\XO}{\ensuremath{X_{\Ord}}}
\newcommand{\Ep}{\ensuremath{\mathrm{E}^+}}
\newcommand{\GOPH}{\ensuremath{\Gamma_{\Ord}^{\text{PH}}}}

\newcommand{\sm}[4]{\ensuremath{\left(\begin{smallmatrix} #1 & #2\\#3 & #4\end{smallmatrix}\right)}}
\newcommand{\lm}[4]{\ensuremath{\left(\begin{matrix} #1 & #2\\#3 & #4\end{matrix}\right)}}
\newcommand{\genmtx}{\sm{a}{b}{c}{d}}
\newcommand{\qord}[1]{\ensuremath{\mathcal{O}_{#1}}} 
\newcommand{\qordd}{\qord{d}} 
\newcommand{\Int}[3]{\ensuremath{\left\langle #1, #2 \right\rangle_{#3}}}
\newcommand{\Intpm}[2]{\Int{#1}{#2}{\pm}}
\newcommand{\tl}[1]{\tilde{\ell}_{#1}}

\newtheorem{theorem}{Theorem}[section]

\newtheorem{lemma}[theorem]{Lemma}
\newtheorem{proposition}[theorem]{Proposition}

\theoremstyle{definition}

\newtheorem{definition}[theorem]{Definition}

\newtheorem{remark}[theorem]{Remark}
\declaretheorem[style=definition, name=Example, numberlike=theorem, qed=$\qedsymbol$]{example}

\numberwithin{equation}{section}

\begin{document}

\title[Intersection series]{Hecke operators acting on optimal embeddings in indefinite quaternion algebras}
\author[J. Rickards]{James Rickards}
\address{University of Colorado Boulder, Boulder, Colorado, USA}
\email{james.rickards@colorado.edu}
\urladdr{https://math.colorado.edu/~jari2770/}
\date{\today}
\thanks{I thank the anonymous reviewer for their extensive list of helpful comments. This research was supported by an NSERC Vanier Scholarship at McGill University. The author is currently partially supported by NSF-CAREER CNS-1652238 (PI Katherine E. Stange).}
\begin{abstract}
We explore a natural action of Hecke operators acting on formal sums of optimal embeddings of real quadratic orders into Eichler orders. By associating an optimal embedding to its root geodesic on the corresponding Shimura curve, we can consider the signed intersection number of pairs of embeddings. Using the Hecke operators and the intersection pairing, we construct a generating series that is demonstrated to be a classical modular form of weight two.
\end{abstract}

\maketitle

\setcounter{tocdepth}{1}
\tableofcontents

\section{Introduction}
Connecting arithmetic generating series to modular forms and L-functions has been an important area of research in the last half century. Early results include the work of Zagier on Hurwitz class numbers and mock modular forms (\cite{Zag75}), and the work of Gross-Kohnen-Zagier on heights of Heegner points (\cite{GKZ87}). Since the 1990's, the aim of the Kudla program has been to investigate relationships between these two worlds.

In \cite{DV20}, Darmon and Vonk proposed a real quadratic analogue of the difference of singular moduli, with the key ingredient being a $p-$adic intersection number of closed geodesics on a Shimura curve. In \cite{JR21shim}, arithmetic properties of these intersections were investigated, and real quadratic analogues of the factorization formulas of $\Nm(j(\tau_1)-j(\tau_2))$ (due to Gross-Zagier in \cite{GZ85}) were derived.

In this paper, we algebraically construct a Hecke operator, which acts on formal sums of optimal embeddings of real quadratic orders into a fixed Eichler order. By considering the signed intersection pairing of the corresponding geodesics on the Shimura curve, we construct a formal generating series, which is proven to be a classical modular of weight $2$ for $\Gamma_0(N)$.

\section{Overview and conventions}

From now on, fix $B$ to be an indefinite quaternion algebra over $\Q$ of discriminant $D$, and let $\Ord$ be an Eichler order in $B$ of level $M$. For an integer $r$, define
\[\Ord^r:=\{z\in \Ord:\nrd(z)=r\},\]
the set of elements in $\Ord$ of reduced norm $r$.

Let $\Emb^+(B)$ denote the set of embeddings of real quadratic fields into $B$. Any $\phi\in\Emb^+(B)$ descends to an optimal embedding of the real quadratic order $\phi^{-1}(\Ord)$ into $\Ord$. We declare $\phi$ and $\psi$ to be equivalent if there exists $x\in\Ord^1$ such that
\[\psi=\phi^x:=x\phi x^{-1}.\]
Denote the equivalence class of $\phi$ by $[\phi]$, and define $\Ep$ to be the set of equivalence classes of elements of $\Emb^+(B)$. Since $\Ord$ is fixed, it is omitted in most of the notation.

\begin{definition}
Let $\phi,\psi\in\Emb^+(B)$, and let $n$ be a positive integer. Define
\[\Theta(n):=\Ord^1\backslash\Ord^n,\]
a finite set.

If $n$ is coprime to $M$, the $n$\textsuperscript{th} weight function associated to $\phi, \psi$ is defined by
\begin{equation}\label{eq:weightdef}
w_n(\phi, \psi):=\left|\left\{\pi\in\Theta(n): \left[\phi^{\pi}\right]=[\psi]\right\}\right|.
\end{equation}
Otherwise, define $w_n(\phi, \psi)=0$.
\end{definition}

The weight function is independent of the choices of $\pi\in\Theta(n)$, is well-defined over the equivalence classes of $\phi$ and $\psi$, but is not symmetric.

\begin{definition}\label{def:hecke}
Let $\phi\in\Emb^+(B)$, and define the Hecke operator $T_n$ acting on $[\phi]$ via the formula
\[T_n[\phi]:=\displaystyle\sum_{[\psi]\in\Ep}w_n(\psi, \phi)[\psi].\]
Extend the definition linearly to $T_n:\Z[\Ep]\rightarrow\Z[\Ep]$.
\end{definition}

See Section \ref{sec:althecke} for a precise description of the discriminants and multiplicities of the terms in $T_p[\phi]$ for $p$ prime.

Fix an embedding of $B$ in $\Mat(2,\R)$, and then
\[\Gamma_{\Ord}:=\Ord^1/\{\pm 1\}\subseteq\PSL(2, \R)\]
is a discrete subgroup, which acts on the upper half plane $\uhp$. The quotient $\YO:=\Gamma_{\Ord}\backslash\uhp$ can be given the structure of a Riemann surface, called a Shimura curve. Elements of $\Ep$ biject with closed geodesics on $\YO$, and there is a well defined notion of the signed intersection number of pairs of closed geodesics. In particular, this extends linearly to the notion of a signed intersection number of two elements $\alpha_1, \alpha_2\in\Z[\Ep]$, which we denote by
\[\Int{\alpha_1}{\alpha_2}{\pm}\in\Z.\]

\begin{definition}
Let $\alpha_1, \alpha_2\in\Z[\Ep]$. The signed intersection series associated to $\alpha_1,\alpha_2$ is the formal power series
\[\IS^{\pm}_{\alpha_1,\alpha_2}(\tau):=\sum_{n=1}^{\infty}\Int{\alpha_1}{T_n\alpha_2}{\pm} q^n,\]
where $q=e^{2\pi i\tau}$.
\end{definition}

Let $S_2(\Gamma_0(n))$ be the set of cusp forms on $\Gamma_0(n)\subseteq\PSL(2, \Z)$, the Hecke congruence subgroup of level $n$. If $m\mid n$, let $S_2(\Gamma_0(n))^{m-\new}\subseteq S_2(\Gamma_0(n))$ be the forms that are new with respect to all prime divisors of $m$.

\begin{theorem}\label{thm:mainthm}
We have $\IS^{\pm}_{\alpha_1,\alpha_2}\in S_2(\Gamma_0(DM^2))^{D-\new}$.
\end{theorem}

In fact, there is a cusp form in $S_2(\Gamma_0(DM))^{D-\new}$ whose coprime to $M$ coefficients match $\IS^{\pm}_{\alpha_1,\alpha_2}$.

Section \ref{sec:background} provides background results on intersection numbers. Sections \ref{sec:atkinlehner}-\ref{sec:heckehomology} cover the main properties of $T_n$, including an alternate expression for $T_n[\phi]$, as well as showing the equivariance of $T_n$ with respect to the signed intersection pairing. Section \ref{sec:mfbackground} covers the main background on quaternionic and classical modular forms, and the pieces are all put together in Section \ref{sec:proof} to prove Theorem \ref{thm:mainthm}. The paper ends in Section \ref{sec:examples}, where we provide a few explicit examples demonstrating the results of Theorem \ref{thm:mainthm} (in particular, $\IS^{\pm}_{\alpha_1,\alpha_2}$ can be non-trivial).

\section{Intersection numbers}\label{sec:background}

The intersection numbers of pairs of optimal embeddings were studied extensively in \cite{JR21shim}, and we recall a few of the basic results.

\begin{definition}
For $\phi\in\Emb^+(B)$, define $d=d(\phi)>0$ to be the discriminant of the quadratic order $\phi^{-1}(\Ord)$. Let $\epsilon_d>1$ be the fundamental unit with positive norm in $\phi^{-1}(\Ord)$.
\end{definition}

The following proposition is classical.

\begin{proposition}\label{prop:primhypbijection}
Define $\GOPH$ to be the set of primitive hyperbolic elements of $\Gamma_{\Ord}$. Then the map $\kappa:\Emb^+(B)\rightarrow\GOPH$ given by
\[\kappa(\phi):=\phi(\epsilon_{d(\phi)})\]
is a bijection. Equivalence classes in $\Emb^+(B)$ correspond to conjugacy classes in $\GOPH$, which are denoted by $C(\GOPH)$.
\end{proposition}

Proposition \ref{prop:primhypbijection} allows us to switch between optimal embeddings and primitive hyperbolic elements, depending on which is more useful at the time.

\begin{definition}\label{def:tildephidef}
Given $\phi\in\Emb^+(B)$, let the (oriented) geodesic in $\uhp$ that connects the two real roots of $\phi(\epsilon_{d(\phi)})$ be $\ell_{\phi}$. This descends to a closed geodesic on $\YO$, denoted by $\tl{\phi}$. Any lift of $\tl{\phi}$ to $\uhp$ is a path between $\tau$ and $\phi(\epsilon_{d(\phi)})\tau$ for some $\tau\in\uhp$.
\end{definition}

The geodesic $\tl{\phi}$ is constant across the equivalence class $[\phi]$. Proposition \ref{prop:intlift} allows the picture to be lifted from $\YO$ to $\uhp$, which will be useful to demonstrate the Hecke-equivariance of the signed intersection pairing.

\begin{proposition}[Proposition 1.4 of \cite{JR21shim}]\label{prop:intlift}
Let $\phi_1,\phi_2\in\Emb^+(B)$. Every transverse intersection point $z\in\tl{\phi_1}\pitchfork\tl{\phi_2}$ lifts to a pair $(\phi_1',\phi_2')\in\Emb^+(B)\times\Emb^+(B)$ with 
\begin{enumerate}[label=(\roman*)]
\item $[\phi_i']=[\phi_i]$ for $i=1,2$;
\item $\ell_{\phi_1'}$ and $\ell_{\phi_2'}$ intersect transversely in $\uhp$.
\end{enumerate}
This association is unique up to the action of simultaneous conjugation by $\Ord^1$ on the pair $(\phi_1',\phi_2')$.
\end{proposition}

\begin{remark}
In this paper, we only consider the signed intersection number. In \cite{JR21shim}, the unsigned and $q-$weighted intersection numbers are also considered, and it is the $q-$weighted intersection that connects to the work of Darmon-Vonk in \cite{DV20}. Determining the nature of the intersection series for these alternate intersection functions is a worthwhile direction for future work.
\end{remark}

\section{Atkin-Lehner operators}\label{sec:atkinlehner}

The theory of Atkin-Lehner operators acting on optimal embeddings is much simpler than the general theory.

\begin{definition}
Let $q$ be prime with $q^e\mid\mid DM$ for some $e\geq 1$. Let $\omega\in\Ord^{q^e}$ be any element that normalizes $\Ord$. Define the Atkin-Lehner involution $W_q$ acting on $[\phi]\in\Ep$ as
\[W_q[\phi]:=[\phi^{\omega}],\]
which is independant of the choice of $\omega$. Extend the action linearly to $\Z[\Ep]$.
\end{definition}

If $q\mid D$, then $|\Theta(q)|=1$, hence $T_q=W_q$. If $q^e\mid\mid M$, then $T_q=0$, so $T_q\neq W_q$. Even if we had extended the weight function definition in Equation \eqref{eq:weightdef} to allow for $q\mid M$, it would still not coincide, as $|\Theta(q^e)|>1$ in general.

\begin{remark}\label{rem:atkinlehnerM}
The reason we take $T_n=0$ when $\gcd(n, M)>1$ is that it is hard to access the $p$\textsuperscript{th} coefficient of an oldform in $S_2(\Gamma_0(DM))^{D-\new}$ when $p\mid M$. Furthermore, this oldform may be combined with newforms in $S_2(\Gamma_0(DM))$, which further muddies the waters. We do not run into this issue with $q\mid D$, as the forms we are studying are guaranteed to be new at such $q$. See Remark \ref{rem:levelDM} for a possible resolution.
\end{remark}

\section{Prime Hecke operators}\label{sec:primehecke}

It will be convenient to sometimes restrict to studying prime Hecke operators. In order to do so, we must show that $T_n$ satisfy the standard recursive formulas. As a first step, we do this for the operators $T'_n$, defined by (for $[\phi]\in\Ep$)
\begin{equation}\label{eqn:naiveheckedef}
T'_n[\phi]:=\sum_{\pi\in\Theta(n)} [\phi^{\pi}]=\sum_{[\psi]\in\Ep}w_n(\phi, \psi)[\psi].
\end{equation}
Aside from some minor details, this is analogous to the classical case.

\begin{lemma}\label{lem:naiveheckeworks}
Let $m,n$ be positive coprime integers, and let $p\nmid DM$ be prime. The following statements are true:
\begin{enumerate}[label=(\roman*)]
\item $T'_{mn}=T'_mT'_n$;
\item $T'_{p^{k}}T'_p=T'_{p^{k+1}}+pT'_{p^{k-1}}$ for all positive integers $k$.
\end{enumerate}
\end{lemma}
\begin{proof}
If either $m$ or $n$ is not coprime to $M$, the first point is trivial. Otherwise, write
\[\Theta(m) = \bigcup_{i=1}^U\Ord^1\pi_i, \qquad \Theta(n) = \bigcup_{i=1}^V\Ord^1\pi'_i.\]
It suffices to show that the set $\{\Ord^1\pi_i\pi'_j\}_{i,j=1}^{i=U,j=V}$ is a valid and complete set of representatives for $\Theta(mn)$. First, if $x\in\Ord^{mn}$, then since $m$ and $n$ are coprime and $\Ord$ has class number $1$, we can write $x=yz$ with $\nrd(y)=m$, $\nrd(z)=n$, and $y,z\in\Ord$ (see Theorem 2.1 of \cite{SC20}). Then $z=u_1\pi'_j$ for some $j$ and $u_1\in\Ord^1$, and $x=(yu_1)\pi'_j$. We have $yu_1=u_2\pi_i$ for some $i$ and $u_2\in\Ord^1$, whence $x=u_2\pi_i\pi'_j$. Thus we have a complete set of representatives for $\Ord^1\backslash\Ord^{mn}$.

To show that they are all distinct, assume otherwise, so that $\Ord^1\pi_i\pi'_j=\Ord^1\pi_{i'}\pi'_{j'}$. Rearranging this gives $\Ord^1 \pi_i\pi'_j\pi_{j'}^{\prime -1} = \Ord^1\pi_{i'}$. If $j=j'$, then $i=i'$ and we are done. Otherwise, let $x=\pi'_j\pi_{j'}^{\prime -1}$; we have $\nrd(x)=1$ and $x\notin\Ord$ since $j\neq j'$. By taking completions, there exists a prime divisor $p$ of $n$ such that $x_p\notin\Ord_p$. Since $\nrd(\pi_i)=m$ is coprime to $n$, it is coprime to $p$, and thus $(\pi_ix)_p\notin\Ord_p$, whence $\pi_ix\notin\Ord$. But $\pi_ix\in\Ord^1\pi_{i'}\subseteq\Ord$, contradiction. Therefore $i=i'$ and $j=j'$, as claimed.

The proof of the second point is exactly the same as for $\SL(2, \Z)$, where explicit representatives of $\Gamma_0(p^e)\backslash\SL(2, \Z)$ can be used.
\end{proof}

We can use these properties of $T_n'$ to deduce them for $T_n$ as well.

\begin{proposition}\label{prop:heckeworks}
Let $m,n$ be positive coprime integers, let $p\nmid DM$ be prime, let $q\mid D$ be prime, and let $k$ be a positive integer. The following statements are true:
\begin{enumerate}[label=(\roman*)]
\item $T_{mn}=T_mT_n$;
\item $T_{p^k}T_p=T_{p^{k+1}}+pT_{p^{k-1}}$;
\item $T_{q^k}=T_q^k$.
\end{enumerate}
\end{proposition}
\begin{proof}
If either $m$ or $n$ is not coprime to $M$, the first point is trivial. Otherwise, Lemma \ref{lem:naiveheckeworks} gives $T'_{mn}[\psi]=T'_n(T'_m[\psi])$. Plug the definition into this expression and match coefficients to obtain
\[w_{mn}(\psi, \phi)=\sum_{[\theta]\in\Ep}w_m(\psi, \theta)w_n(\theta, \phi),\]
for all $[\psi]\in\Ep$. By expanding out $T_{mn}[\phi]$ and $T_m(T_n[\phi])$ in a similar fashion, this implies that they are equal, and the point follows.

The second point follows in an analogous fashion. The third point follows from $|\Theta(q^k)|=1$.
\end{proof}

\section{Alternate expression for Hecke operators}\label{sec:althecke}

We will require an alternate expression for $T_p$, as summing over all embedding classes is not convenient. Such an expression is given in Proposition \ref{prop:alternateheckeexpression}, and the rest of this section is spent proving it.

\begin{proposition}\label{prop:alternateheckeexpression}
Let $p\nmid DM$ be a prime and $\phi\in\Emb^+(B)$. Then
\[T_p[\phi]=\sum_{\pi\in\Theta(p)}\dfrac{\log{\epsilon_{d(\phi)}}}{\log{\epsilon_{d(\phi^{\pi})}}}[\phi^{\pi}].\]
\end{proposition}

Proposition \ref{prop:alternateheckeexpression} is also true with $p$ replaced by an integer coprime to $M$, though we will not need this level of generality.

To begin, we first shift the definition of $T_p$ to being over $\pi\in\Theta(p)$, and not over $[\psi]\in\Ep$.

\begin{lemma}\label{lem:firstalternate}
Let $p\nmid DM$ be a prime and $\phi\in\Emb^+(B)$. Then
\[T_p[\phi]=\sum_{\pi\in\Theta(p)}\dfrac{w_p(\phi^{\pi}, \phi)}{w_p(\phi, \phi^{\pi})}[\phi^{\pi}].\]
\end{lemma}
\begin{proof}
Recall that
\[T_p[\phi]=\displaystyle\sum_{[\psi]\in\Ep}w_p(\psi, \phi)[\psi].\]
It is clear that $w_p(\psi, \phi)>0$ if and only if $w_p(\phi, \psi)>0$, so every embedding $[\psi]$ with non-zero coefficient must be of the form $[\phi^{\pi}]$ for some $\pi\in\Theta(p)$. Thus we can rewrite $T_p$ by summing over $[\phi^{\pi}]$, making sure to divide by the amount we overcount each $[\phi^{\pi}]$. This proves the lemma.
\end{proof}

By Lemma \ref{lem:firstalternate}, it suffices to show that $w_p(\phi, \phi^{\pi})$, relates to fundamental units.

\begin{definition}
If $d$ is a discriminant and $p$ is a prime such that $\frac{d}{p^2}$ is not a discriminant, we say $d$ is $p-$fundamental.
\end{definition}

The following proposition describes the behaviour of $[\phi^{\pi}]$ for $\pi\in\Theta(p)$.

\begin{proposition}\label{prop:heckedivision}
Let $\phi\in\Emb^+(B)$ correspond to an optimal embedding of discriminant $d=d(\phi)$. Let $p$ be a prime with $p\nmid DM$, and write $d=p^{2k}d'$, where $d'$ is a $p-$fundamental discriminant. Consider the multiset of $p+1$ optimal embeddings classes corresponding to $\{[\phi^{\pi}]:\pi\in\Theta(p)\}$. This contains
\begin{itemize}
\item $p+1$ optimal embeddings of discriminant $p^2d$ if $k=0$ and $\left(\frac{d}{p}\right)=-1$.
\item $p$ optimal embeddings of discriminant $p^2d$ and one of discriminant $d$ if $k=0$ and $\left(\frac{d}{p}\right)=0$.
\item $p-1$ optimal embeddings of discriminant $p^2d$ and two of discriminant $d$ if $k=0$ and $\left(\frac{d}{p}\right)=1$.
\item $p$ optimal embeddings of discriminant $p^2d$ and one of discriminant $\frac{d}{p^2}$ if $k>0$.
\end{itemize}
Let $\epsilon_{p^2d}=\epsilon_d^r$, and the optimal embeddings of discriminant $p^2d$ divide into $\frac{p-\left(\frac{d}{p}\right)}{r}$ distinct equivalence classes, each with multiplicity $r$.
\end{proposition}
\begin{proof}
Assume that $p$ is odd; $p=2$ is covered in Proposition \ref{prop:heckedivision2}. To compute the discriminants, it suffices to work in the completion at $p$. In particular, we can assume that $\Ord_p=\Mat(2,\Z_p)$, $\phi_p(\sqrt{d})=\sm{0}{d}{1}{0}$, and we can take the following as representatives for $\Theta(p)$:
\[\pi_i=\lm{1}{i}{0}{p}:i=0,1,\cdots,p-1,\qquad\pi_{\infty}=\lm{p}{0}{0}{1}.\]
We compute
\[\phi^{\pi_{\infty}}(\sqrt{d})=\lm{0}{pd}{\frac{1}{p}}{0},\]
which is an optimal embedding of $p^2d$ in all cases. For $i<\infty$,
\[\phi^{\pi_i}(\sqrt{d})=\lm{i}{\frac{d-i^2}{p}}{p}{-i}.\]
If $p$ is inert with respect to $d$, then $\frac{d-i^2}{p}\notin\Z_p$, always giving discriminant $p^2d$. If $p\mid d$, we have discriminant $p^2d$ for all $i\neq 0$. When $i=0$, we get $d$ or $\frac{d}{p^2}$, depending on if $p^2\mid d$ or not. Finally, if $p$ is split with respect to $d$, then precisely two values of $i$ allow this to lie in $\Mat(2,\Z_p)$, and we get $2$ embeddings of discriminant $d$ and $p-2$ of $p^2d$. Therefore the discriminants occur as claimed.

Next, we check when we get similar embeddings of discriminant $p^2d$. Let $v=\phi(\epsilon_d)\in\Ord^1$, fix $i$, and let $\pi_iv=u\pi_j$ for some (unique) $j$ and $u\in\Ord^1$. Then
\[\phi^{\pi_j}=\pi_j\phi \pi_j^{-1}\sim u\pi_j\phi \pi_j^{-1}u^{-1}=\pi_iv\phi v^{-1}\pi_i^{-1}=\pi_i\phi \pi_i^{-1}=\phi^{\pi_i},\]
i.e. the resulting forms lie in the same equivalence class. We wish to show that in the discriminant $p^2d$ case this is also essentially necessary, i.e. if $\phi^{\pi_i}\sim\phi^{\pi_j}$ are embeddings of discriminant $p^2d$, then $\pi_iv^k=u\pi_j$ for some integer $k$ and $u\in\Ord^1$.

Indeed, $[\phi^{\pi_i}]=[\phi^{\pi_j}]$ if and only if there is a $u\in\Ord^1$ for which $\pi_i\phi \pi_i^{-1}=u^{-1}\pi_j\phi \pi_j^{-1}u$. Rearranging, this is equivalent to
\[\pi_j^{-1}u\pi_i\phi(\sqrt{d})(\pi_j^{-1}u\pi_i)^{-1}=\phi(\sqrt{d}).\]
In particular, $\pi_j^{-1}u\pi_i$ normalizes $\phi(\sqrt{d})$, and it follows from Proposition 7.7.8 of \cite{JV21} that
\[\pi_j^{-1}u\pi_i=\phi(x+y\sqrt{d})\text{ for $x,y\in\Q$.}\]
After rearranging, this is equivalent to $\pi_j\phi(x+y\sqrt{d})\pi_i^{-1}\in\Ord^1$. Taking norms, we see that $x^2-dy^2=1$, whence we are done if we can show that $z=x+y\sqrt{d}\in\qordd$, the order of discriminant $d$. Since $\phi(pz)=\overline{\pi_j}u\pi_i\in\Ord$, we have $z\in\frac{1}{p}\mathcal{O}_d$, and it suffices to look at the completion at $p$.

In this completion, we can take the explicit forms of $\pi_i$ and $\phi$ as above. Thus $\phi(x+y\sqrt{d})=\sm{x}{yd}{y}{x}$. If $i,j<\infty$, then
\[\pi_j\phi(x+y\sqrt{d})\pi_i^{-1}=\lm{x+jy}{\frac{(j-i)x+(d-ij)y}{p}}{py}{x-iy}\in \Mat(2,\Z_p).\]
From above, $px,py\in\Z _p$, so write $X=px,Y=py$. Then 
\[p\mid X+jY,\qquad p^2\mid (j-i)X+(d-ij)Y,\]
and looking at the second equation modulo $p$, we derive
\[0\equiv (j-i)(-jY)+(d-ij)Y\equiv (d-j^2)Y\pmod{p}.\]
Since we have embeddings of discriminant $p^2d$, $d-j^2\not\equiv 0\pmod{p}$, whence $p\mid Y$, and so $p\mid X$, as desired.

If $i=\infty$, then $j<\infty$, and we have
\[\pi_j\phi(x+y\sqrt{d})\pi_{\infty}^{-1}=\lm{\frac{x+jy}{p}}{yd+jx}{y}{px}\in \Mat(2,\Z_p).\]
It immediately follows that $y\in\Z _p$, and then $x\in\Z_p$ too, as desired.

Now, we see that we form equivalence classes by right multiplication by $v=\phi(\epsilon_d)$. Thus the size of an orbit corresponds to the minimal $k$ such that $\pi_iv^k=u\pi_i$, for some $u\in\Ord^1$. Writing $v^k=\phi(X+Y\sqrt{d})$, in the above calculations we can take $i=j$ (as well as repeating for $i=j=\infty$), and it follows that $\pi_iv^k=u\pi_i$ if and only if $p\mid Y$. The smallest such $k$ is $k=r$, since $p\mid Y$ is equivalent to $X+Y\sqrt{D}\in\qord{p^2d}$.
\end{proof}

When $p=2$, the above proof needs to be modified a bit. For sake of clarity, we restate the proposition explicitly before giving the proof.

\begin{proposition}\label{prop:heckedivision2}
Let $\phi\in\Emb^+(B)$ correspond to an optimal embedding of discriminant $d=d(\phi)$. Assume $2\nmid DM$, and write $d=2^{2k}d'$, where $d'$ is a $2-$fundamental discriminant. Consider the multiset of $3$ optimal embeddings classes corresponding to $\{[\phi^{\pi}]:\pi\in\Theta(2)\}$. This contains
\begin{itemize}
\item $3$ optimal embeddings of discriminant $4d$ if $k=0$ and $d\equiv 5\pmod{8}$.
\item $2$ optimal embeddings of discriminant $4d$ and $1$ of discriminant $d$ if $k=0$ and $d\equiv 0\pmod{2}$.
\item $1$ optimal embedding of discriminant $4d$ and $2$ of discriminant $d$ if $k=0$ and $d\equiv 1\pmod{8}$.
\item $2$ optimal embeddings of discriminant $4d$ and $1$ of discriminant $\frac{d}{4}$ if $k>0$.
\end{itemize}
Let $\epsilon_{4d}=\epsilon_d^r$, and the optimal embeddings of discriminant $4d$ divide into $\frac{2-\left(\frac{d}{2}\right)}{r}$ distinct equivalence classes, each with multiplicity $r$.
\end{proposition}
\begin{proof}
We mostly mirror the proof of Proposition \ref{prop:heckedivision}. We can work locally, so that $\Ord_2=\Mat(2,\Z_2)$, and we can assume that
\[\phi_2(\sqrt{d})=\lm{p_d}{(d-p_d)/2}{2}{-p_d},\]
where $p_d$ is the parity of $d$. We can take representatives for $\Theta(2)$ as
\[\pi_i=\lm{1}{i}{0}{2}:i=0,1,\qquad\pi_{\infty}=\lm{2}{0}{0}{1}.\]
We compute
\[\phi^{\pi_{\infty}}(\sqrt{d})=\lm{p_d}{d-p_d}{1}{-p_d},\]
which is an optimal embedding of discriminant $4d$. For $i=0,1$,
\[\phi^{\pi_i}(\sqrt{d})=\lm{p_d+2i}{\frac{d-p_d}{4}-p_di-i^2}{4}{-2i-p_d}.\]
If $d\equiv 5\pmod{8}$, the top right coefficient is odd for $i=0,1$, whence this is an optimal embedding of discriminant $4d$. If $d\equiv 1\pmod{8}$, these are optimal of discriminant $d$ for $i=0,1$. Finally, if $d$ is even, then the top right coefficient is $d/4-i^2$ which is odd and even for the two choices of $i$. Since all other coefficients are even, this will be an optimal embedding of discriminant $4d$ for exactly one of the two choices of $i$, and an embedding of discriminant $d$ for the other. The only way the embedding of discriminant $d$ is not optimal is if either $16\mid d$ and $i$ is even, or $d\equiv 4\pmod{16}$ and $i=1$. In both of these cases the embedding is optimal of discriminant $d/4$, and these cases are equivalent to $k>0$. Therefore the discriminants occur as claimed.

Next, we check when we get similar embeddings of discriminant $4d$. Let $v=\phi(\epsilon_d)\in\Ord^1$, fix $i$, and let $\pi_iv=u\pi_j$ for some $j$ and $u\in\Ord^1$. As before, $\phi^{\pi_j}\sim\phi^{\pi_i}$, and we want to show that if this equation holds then $\pi_iv^k=u\pi_j$ for some integer $k$ and $u\in\Ord^1$.

As in Proposition \ref{prop:heckedivision}, this rearranges to $\pi_j\phi(x+y\sqrt{d})\pi_i^{-1}=u\in\Ord^1$ for some rationals $x,y$. Taking norms, $x^2-dy^2=1$, whence we are done if we can show that $z=x+y\sqrt{d}\in\qordd$. As $\phi(2z)=\overline{\pi_j}u\pi_i\in\Ord$, we have $z\in\frac{1}{2}\mathcal{O}_d$, hence $4x,4y\in\Z$. Take the explicit forms of $\pi_i$ and $\phi$ as above; in particular,
\[\phi(x+y\sqrt{d})=\lm{x+p_dy}{y(d-p_d)/2}{2y}{x-p_dy}.\]
If $i,j\in\{0,1\}$, we can assume they are distinct, hence $i=0,j=1$, and $d\equiv 5\pmod{8}$ (as the embeddings have discriminant $4d$). Then
\[\pi_j\phi(x+y\sqrt{d})\pi_i^{-1}=\lm{x+3y}{\frac{1}{2}x+\frac{d-3}{4}y}{4y}{x-y}\in\Mat(2,\Z_2).\]
Write $4x=X$ and $4y=Y$, and this implies that
\[X\equiv Y\pmod{4},\qquad 2X+(d-3)Y\equiv 0\pmod{16},\qquad X^2-dY^2=16.\]
If $X$ is odd, then $Y$ is odd, hence $0\equiv X^2-dY^2\equiv 1-5\equiv 4\pmod{8}$, contradiction. Thus $X,Y$ are even, and $X/2\equiv Y/2\pmod{2}$. Since $z=\frac{(X/2)+(Y/2)\sqrt{d}}{2}$, this implies that $z\in\qordd$, as required.

If $i=\infty$ and $j=0,1$, we have
\[\pi_j\phi(x+y\sqrt{d})\pi_{\infty}^{-1}=\lm{\frac{x}{2}+\frac{(p_d+2j)y}{2}}{jx+((d-p_d)/2-p_dj)y}{2y}{2x-2yp_d}\in \Mat(2,\Z_2).\]
Thus $2x,2y\in\Z$, write $2x=X$ and $2y=Y$, and it requires to show that $X\equiv Yd\pmod{2}$. But $X^2-dY^2=4$, so the conclusion follows.

The finish is exactly as in Proposition \ref{prop:heckedivision}.
\end{proof}

We can now prove the alternate expression for $T_p$.

\begin{proof}[Proof of Proposition \ref{prop:alternateheckeexpression}]
By Lemma \ref{lem:firstalternate},
\[T_p[\phi]=\sum_{\pi\in\Theta(p)}\dfrac{w_p(\phi^{\pi}, \phi)}{w_p(\phi, \phi^{\pi})}[\phi^{\pi}].\]
Let $d=d(\phi)$, and by Proposition \ref{prop:heckedivision}, the terms $[\phi^{\pi}]$ all have discriminant $p^2d, d, d/p^2$.

Start with the terms having discriminant $p^2d$. Let $\epsilon_{p^2d}=\epsilon_d^r$, and then Proposition \ref{prop:heckedivision} says that $w_p(\phi, \phi^{\pi})=r$. Similarly, $w_p(\phi^{\pi}, \phi)=1$, as we decreased the discriminant. Therefore we have
\[\dfrac{w_p(\phi^{\pi}, \phi)}{w_p(\phi, \phi^{\pi})}=\dfrac{1}{r}=\dfrac{\log{\epsilon_{d(\phi)}}}{\log{\epsilon_{d(\phi^{\pi})}}},\]
as desired.

For the terms of discriminant $d$, Proposition \ref{prop:heckedivision} implies that
\[w_p(\phi^{\pi}, \phi)=w_p(\phi, \phi^{\pi})\in\{1, 2\},\]
as desired.

Finally, the terms of discriminant $d/p^2$ can be handled analogously to $p^2d$, completing the proof.
\end{proof}

\section{Hecke operators acting on homology}\label{sec:heckehomology}

For the rest of this paper, assume that $D>1$ (see Remark \ref{rem:m2q} for changes to the $D=1$ case). In this case, there are no cusps, and $\XO:=\Gamma_{\Ord}\backslash\overline{\mathbb{H}}=\YO$. We can transfer the Hecke operators to act on homology via the association of $\phi\rightarrow\tl{\phi}$ from Definition \ref{def:tildephidef}. It is useful to switch from $\C[\Ep]$ to $\C[C(\GOPH)]$, which is accomplished through the bijection $\kappa$ from Proposition \ref{prop:primhypbijection}.

\begin{definition}
Let $\gamma\in\Gamma_{\Ord}$, and denote by $\tl{\gamma}\in H_1(\XO, \C)$ the image of the geodesic between $\tau$ and $\gamma\tau$, which is independant of $\tau\in\uhp$.
\end{definition}

Consider the map $\eta:\C[C(\GOPH)]\rightarrow H_1(\XO, \C)$ induced by
\[\gamma\rightarrow \tl{\gamma},\]
where $\gamma\in\GOPH$.

\begin{lemma}\label{lem:alphasurjective}
The elements $\tl{\gamma}$ for $\gamma\in\GOPH$ generate $H_1(\XO, \C)$. In particular, $\eta$ is surjective.
\end{lemma}
\begin{proof}
The lemma is clearly true if we allow all $\gamma\in\Gamma_{\Ord}$. We may restrict to $\gamma\in\GOPH$ since the elements pairing the sides of a Dirichlet domain for $\Gamma_{\Ord}$ are primitive, hyperbolic, and generate $\Gamma_{\Ord}$ (see \cite{JV09}).
\end{proof}

Allowing for norm one elements that have a power in $\Ord$ will make our life easier, as conjugation of elements of $\GOPH$ by $\pi\in\Theta(p)$ will produce such elements. To this end, we have the following definition.

\begin{definition}
Let $\gamma\in B^{\times}$, and assume that $\gamma^r\in\Gamma_{\Ord}$ for some $r\in\Z^+$. Define
\[\tl{\gamma}:=\dfrac{1}{r}\tl{\gamma^r}\in H_1(\XO, \C),\]
which is independant of $r$.
\end{definition}

With this convention, the induced action of Hecke operators on $H_1(\XO, \C)$ takes a particularly nice form.

\begin{proposition}\label{prop:heckeonhomology}
Let $p\nmid DM$ be a prime and let $\gamma\in\GOPH$. Then
\[\eta(T_p[\gamma])=\sum_{\pi\in\Theta(p)}\tl{\pi\gamma\pi^{-1}}.\]
\end{proposition}
\begin{proof}
By Proposition \ref{prop:primhypbijection}, write $\gamma=\phi(\epsilon_d)$ for some $\phi\in\Emb^+(B)$ and $d=d(\phi)$. From Proposition \ref{prop:alternateheckeexpression},
\[\eta(T_p[\gamma])=\sum_{\pi\in\Theta(p)}\dfrac{\log{\epsilon_{d}}}{\log{\epsilon_{d(\phi^{\pi})}}}\eta([\phi^{\pi}]).\]

For $\pi\in\Theta(p)$, let $d'=d(\phi^{\pi})$. Note that $\epsilon_{d'}=\epsilon_d^r$ for some $r$ that is either an integer or the reciprocal of an integer, since $\phi^{-1}(\Ord)$ and $\left(\phi^{\pi}\right)^{-1}(\Ord)$ are orders in the same quadratic field. In particular,
\[\phi^{\pi}(\epsilon_{d'})=\pi\phi(\epsilon_d^r)\pi^{-1}=(\pi\gamma\pi^{-1})^r,\]
hence
\[\eta([\phi^{\pi}])=\tl{(\pi\gamma\pi^{-1})^r}=r\tl{\pi\gamma\pi^{-1}}.\]
The coefficient
\[\dfrac{\log{\epsilon_{d}}}{\log{\epsilon_{d'}}}=\dfrac{1}{r},\]
which cancels with $r$, giving the result.
\end{proof}

In order to prove that the signed intersection pairing is Hecke-equivariant, we shift back to the original definition of $T_n$.

\begin{proposition}\label{prop:heckeselfadjoint}
For all positive integers $n$ and $\alpha_1,\alpha_2\in\C[\Ep]$, we have
\[\Intpm{T_n\alpha_1}{\alpha_2}=\Intpm{\alpha_1}{T_n\alpha_2}\]
\end{proposition}
\begin{proof}
It suffices to prove this proposition for $n=p\nmid M$ a prime and $\alpha_i=[\phi_i]\in\Emb^+(\Ord)$. Write $\Theta(p)=\displaystyle\cup_{i=1}^{U}\Ord^1\pi_i$ ($U=1$ if $p\mid\mathfrak{D}$ and $=p+1$ otherwise), and define the set
\[S_1=\{(\pi_i,u,\phi):|\ell_{\phi}\pitchfork\ell_{\phi_2}|=1, \phi^{\pi_i}=\phi_1^{u},u\in\Ord^1\}.\]
Use
\[T_p[\phi_1]=\sum_{[\psi]}w_p(\psi,\phi_1)[\psi],\]
and expand out $\Intpm{T_p[\phi_1]}{[\phi_2]}$. We claim that each term corresponds to an element of $S_1$. 

By Proposition \ref{prop:intlift}, an intersection of $[\psi]$ with $[\phi_2]$ corresponds to the simultaneous equivalence class of the pair $(\psi^v,\phi_2)$ with $v\in\Ord^1$ and $|\ell_{\psi^v}\pitchfork\ell_{\phi_2}|=1$. Note that $w_p(\psi,\phi_1)=w_p(\psi^v,\phi_1)$, so for each of the $w_p(\psi, \phi_1)$ values of $i$ such that $\psi^{\pi_iv}=\phi_1^u$ with $u\in\Ord^1$, we associate the triple
\[(\pi_i, u, \psi^v)\in S_1\]
to the intersection.

Since there were several choices made, we want to determine all possible triples associated to an intersection in $S_1$, so that we can create a bijection with a quotient of $S_1$ by an equivalence relation. Write $r_i=\phi_i(\epsilon_{d(\phi_i)})$ for $i=1,2$, and then the pair $(\psi^v,\phi_2)$ is well defined up to simultaneous conjugation by powers of $r_2$. Furthermore, $u$ is defined up to multiplication on the right by powers of $r_1$. In particular, let $k_1,k_2\in\Z$, write
\[\pi_ir_2^{-k_2}=\delta_i\pi_{i^*},\]
for a unique $\pi_{i^*}$ and $\delta_i\in\Ord^1$, and define an equivalence relation on $S_1$ via
\[(\pi_i,u,\phi)\sim_{S_1}(\pi_{i^*},\delta_i^{-1}ur_1^{k_1},\phi^{r_2^{k_2}}).\]
This relation corresponds exactly to the ambiguity described above in associating an element of $S_1$ to $\Intpm{T_p([\phi_1])}{[\phi_2]}$. Therefore
\[\text{Intersections of $T_p[\phi_1]$ with $[\phi_2]$}\Leftrightarrow\text{ $S_1/\sim_{S_1}$}.\]

Define $S_2$ and the equivalence relation $\sim_{S_2}$ in the analogous fashion, i.e. with all indices $1,2$ swapped. In the exact same manner, we have that intersections of $[\phi_1]$ with $T_p[\phi_2]$ biject naturally with $S_2/\sim_{S_2}$.

Let $(\pi_i,u,\phi)\in S_1$, and let $j,v$ be uniquely defined so that
\[p\pi_i^{-1}u=v^{-1}\pi_j,\]
where $v\in\Ord^1$. We define the map $\theta:S_1\rightarrow S_2$ via
\[\theta((\pi_i,u,\phi))=\left(\pi_j,v,\phi_2^{\pi_j^{-1}v}\right).\]
First, we check that the image lands in $S_2$. Use the shorthand notation $(\psi_1,\psi_2)$ for ``$\ell_{\psi_1}$ and $\ell_{\psi_2}$ intersect transversely,'' and since M\"{o}bius maps preserve intersection,
\[(\phi,\phi_2)\Rightarrow \left(\phi^{u^{-1}\pi_i},\phi_2^{u^{-1}\pi_i}\right)=\left(\phi_1,\phi_2^{\pi_j^{-1}v}\right).\]
Since $(\phi_2^{\pi_j^{-1}v})^{\pi_j}=\phi_2^v$, the image lands in $S_2$. The sign of the intersection is also preserved, since $\nrd(u^{-1}\pi_i)=p>0$.

Let $\theta':S_2\rightarrow S_1$ be the analogously defined map going the other way (swap $1$'s and $2$'s), and it is straightforward to check that $\theta,\theta'$ are inverses to each other, whence $S_1$ bijects with $S_2$. To complete the proposition, it suffices to check that $\theta$ descends to a map from $S_1/\sim_{S_1}$ to $S_2/\sim_{S_2}$ (the map $\theta'$ will do the same in analogous fashion).

Take the equations
\begin{align*}
\pi_ir_2^{-k_2}= & \delta_i\pi_{i^*}, & p\pi_i^{-1}u= & v^{-1}\pi_j, & p\pi_{i^*}^{-1}\delta_i^{-1}ur_1^{k_1} = & v'^{-1}\pi_{j'};\\
\theta((\pi_i,u,\phi))= & \left(\pi_j,v,\phi_2^{\pi_j^{-1}v}\right), & & & \theta((\pi_{i^*},\delta_i^{-1}ur_1^{k_1},\phi^{r_2^{k_2}}))= & \left(\pi_{j'},v',\phi_2^{\pi_{j'}^{-1}v'r_2^{k_2}}\right),
\end{align*}
and we need to show that the right hand side of the bottom two equations are equivalent under $S_2$. Rearranging the above equations gives
\begin{align*}
\pi_{j'}r_1^{-k_1}= & v'p(\pi_{i^*}^{-1}\delta_i^{-1})u\\
= & v'r_2^{k_2}(p\pi_i^{-1}u)\\
= & (v'r_2^{k_2}v^{-1})\pi_j.
\end{align*}
Therefore
\begin{align*}
\left(\pi_{j'},v',\phi_2^{\pi_{j'}^{-1}v'r_2^{k_2}}\right)\sim_{S_2} & \left(\pi_j,(v'r_2^{k_2}v^{-1})^{-1}v'r_2^{k_2},\phi_2^{r_1^{k_1}\pi_{j'}^{-1}v'r_2^{k_2}}\right)\\
= & \left(\pi_j,v,\phi_2^{\pi_j^{-1}v}\right),
\end{align*}
as claimed.
\end{proof}

\begin{remark}\label{rem:qweightalsoworks}
With a closer analysis, one can show that the Hecke operators are also equivariant for the $q-$weighed intersection number (considered in \cite{JR21shim}) when $q\mid DM$. If $q\nmid DM$, then this may fail for the operator $T_q$.
\end{remark}

\section{Modular form background}\label{sec:mfbackground}

Before delving into the proof of Theorem \ref{thm:mainthm}, we recall the relevant bits of quaternionic and classical modular form theory.

Our reference for quaternionic modular forms is sections 3, 5 of \cite{DemV13}, and Sections 2, 3 of \cite{GV11}. For uniformity of presentation, assume that $D>1$ and $B$ is embedded in $\Mat(2, \R)$.

For $\gamma=\genmtx\in B^{\times}/\{\pm 1\}$ and a holomorphic function $f:\mathbb{H}\rightarrow\C$, define the (weight two) slash operator as
\[(f|\gamma)(z):=\det(\gamma)(cz+d)^{-2}f(\gamma z).\]

\begin{definition}
A quaternionic modular form of weight $2$ and level $M$ for $B$ is a holomorphic function $\mathbb{H}\rightarrow\C$ such that
\[(f|\gamma)(z)=f(z)\]
for all $\gamma\in\Gamma_{\Ord}$.

Let $\text{M}_2^{B}(\Ord)$ denote the space of weight two quaternionic modular forms with respect to $\Ord$, and $S_2^{B}(\Ord)$ the subset of cusp forms. Since there are no cusps, all quaternionic modular forms are cusp forms.
\end{definition}

It is possible to define Hecke operators acting on quaternionic modular forms. Integration gives the Hecke-equivariant Eichler-Shimura isomorphism to the dual of the homology:
\[S_2^{B}(\Ord)\oplus\overline{S_2^{B}(\Ord)}\xrightarrow[]{\sim}H_1(\XO, \C)^*.\]
The connection to classical modular forms comes from the Hecke-equivariant Jacquet-Langlands correspondence:
\[S_2^{B}(\Ord)\simeq S_2(\Gamma_0(DM))^{D-\new}.\]

In order to connect coefficients of forms in $S_2(\Gamma_0(DM))^{D-\new}$ to Hecke operators, we recall a few of the main results of Atkin and Lehner. A special case of Theorem 3 of \cite{AL70} is the following proposition.

\begin{proposition}\label{prop:classicalfcoeffs}
Let $f(\tau)=\sum_{n=1}^{\infty}a_nq^n$ be a weight $2$ newform on $\Gamma_0(N)$, normalized so that $a_1=1$. Then
\begin{enumerate}[label=(\roman*)]
\item If $p$ is a prime with $p\nmid N$, then 
\begin{enumerate}
\item $f|T_p=a_pf$;
\item $a_{np}=a_na_p-pa_{n/p}$ for all $n\geq 1$, with $a_{n/p}=0$ if $p\nmid n$.
\end{enumerate}
\item If $q$ is a prime with $q^e\mid\mid N$ for some $e>0$, then
\begin{enumerate}
\item $f|W_q=\lambda(q)f$, where $\lambda(q)=\pm 1$.
\item $a_{nq}=a_na_q$ for all $n\geq 1$;
\item If $e\geq 2$, then $a_q=0$;
\item If $e=1$, then $a_q=-\lambda(q)$, hence $f|W_{q}=-a_qf$.
\end{enumerate}
\end{enumerate}
\end{proposition}

\begin{remark}\label{rem:JLatkinlehner}
In the Jacquet-Langlands correspondence, the Atkin-Lehner operators $W_q$ for $q\mid D$ acting on Shimura curves in fact pick up the Eigenvalue $a_q$, and not $\lambda(q)=-a_q$ (see Theorem 1.2 of \cite{BD96}). This is why we did not need to negate the definition of $W_q=T_q$ acting on optimal embeddings!
\end{remark}

We will be working with the space $S_2(\Gamma_0(DM))^{D-\new}$, hence if $M\neq 1$ we also need to work with oldforms. Theorem 5 of \cite{AL70} provides the description of the new and oldforms, restated as follows.

\begin{proposition}\label{prop:classmfdecomp}
The space $S_2(\Gamma_0(N))$ has a basis which is a direct sum of classes, which consist of newclasses and oldclasses. Every form in a class has the same eigenvalues for $T_p$ with $p$ a prime not dividing $N$, and forms in different classes have distinct eigenvalues at $T_p$ for infinitely many primes $p$. Each newclass consists of a single form, which is an eigenform for all $T_p$ and $W_q$. Each oldclass consists of a set of forms $\{f(d\tau)\}$, where $f\in S_2(\Gamma_0(N'))^{\new}$ for some $N'$ dividing $N$ properly, and $d$ ranges over all positive divisors of $N/N'$. Furthermore, any such set is an oldclass.
Each oldclass can be given an alternate basis where the forms are also eigenforms for all $W_q$.
\end{proposition}

While we can access the $qn^{\text{th}}$ Fourier coefficients of an eigenform in $S_2(\Gamma_0(N))$ with $q\mid N$, it requires knowing which oldclass the form belongs to. If we have no a priori knowledge of this, then the task is less feasible. Since Jacquet-Langlands can produces $M-$old forms, we treat this issue by ignoring coefficients that are not coprime to $M$.

\section{Proof of modularity}\label{sec:proof}

Let $\beta$ be the isomorphism from $H_1(\XO, \C)$ to its dual $H_1(\XO, \C)^*$ induced by the (nondegenerate) signed intersection pairing, i.e.
\[\beta(\psi)(\psi'):=\Intpm{\psi}{\psi'},\]
for $\psi,\psi'\in H_1(\XO, \C)$. The action of the Hecke operators on $H_1(\XO, \C)^*$ is given by Section 5 of \cite{DemV13}. First, let $p\nmid DM$, and write
\[\Theta(p)=\bigcup_{i=1}^{p+1}\Ord^1\pi_i.\]
Let $\gamma\in\Gamma_{\Ord}$, and multiplication on the right by $\gamma$ permutes $\Theta(p)$. Therefore there is a unique permutation $\gamma^{*}$ of $\{1,2,\ldots,p+1\}$ for which
\[\pi_a\gamma=\delta_a\pi_{\gamma^{*}a},\]
for some $\delta_a\in\Gamma_{\Ord}$. The operator $T_p$ is given by
\[T_p(f)(\tl{\gamma}):=\sum_{\pi_a\in\Theta(p)}f(\tl{\delta_a}),\]
for $f\in H_1(\XO, \C)^*$.

Similarly, if $q^e\mid\mid DM$, the Atkin-Lehner operator $W_q$ is given by
\[W_q(f)(\tl{\gamma}):=f(\ell_{\gamma^{\omega}}),\]
where $\omega\in\Ord^{q^e}$ normalizes $\Ord$.

The composition $\beta\circ\eta$ is a map from $\C[C(\GOPH)]$ to $H_1(\XO, \C)^*$, with Hecke operators defined on each end.

\begin{lemma}
The map $\beta\circ\eta$ is Hecke-equivariant for $T_p$ with $p\nmid DM$ and $W_q$ for $q\mid D$.
\end{lemma}
\begin{proof}
The case of $W_q$ for $q\mid D$ follows directly from the definitions.

Next, consider $T_p$ for $p\nmid DM$. Let $\sigma,\gamma\in\GOPH$, and then
\[T_p(\beta\circ\eta([\sigma]))(\tl{\gamma})=\sum_{\pi_a\in\Theta(p)}\Intpm{\tl{\sigma}}{\tl{\delta_a}},\]
where $\pi_a\gamma=\delta_a\pi_{\gamma^{*}a}$.

Applying $T_p$ to $[\phi]$ first gives
\begin{align*}
\beta\circ\eta(T_p([\sigma]))(\tl{\gamma})= & \Intpm{\eta(T_p[\sigma])}{\eta([\gamma])}\\
= & \Intpm{\eta([\sigma])}{\eta(T_p[\gamma])}\\
= & \sum_{\pi_a\in\Theta(p)}\Intpm{\tl{\sigma}}{\tl{\pi_a\gamma\pi_a^{-1}}},
\end{align*}
where we used Propositions \ref{prop:heckeselfadjoint} and \ref{prop:heckeonhomology} in the second and third lines respectively. Thus it suffices to prove that in homology,
\[\sum_{\pi_a\in\Theta(p)}\tl{\delta_a}=\sum_{\pi_a\in\Theta(p)}\tl{\pi_a\gamma\pi_a^{-1}}.\]

Consider $\delta_a=\pi_a\gamma\pi_{\gamma^{*}a}^{-1}$, and note that if $a_1,a_2,\ldots,a_r$ is a sequence, then
\[\sum_{i=1}^r\tl{\delta_{a_i}}=\tl{\delta_{a_1}\delta_{a_2}\cdots\delta_{a_r}}.\]
Decompose the permutation $\gamma^{*}$ into cycles, and say $(a_1,a_2,\ldots,a_r)$ is one such cycle. The intermediate terms all cancel, and we derive
\[\delta_{a_1}\delta_{a_2}\cdots\delta_{a_r}=\pi_{a_1}\gamma^r\pi_{a_1}^{-1}.\]
Therefore
\[\sum_{i=1}^r\tl{\delta_{a_i}}=r\tl{\pi_{a_1}\gamma\pi_{a_1}^{-1}}.\]
Repeat this for all cyclic shifts of $(a_1,a_2,\ldots,a_r)$ to derive
\[\sum_{i=1}^r\tl{\delta_{a_i}}=\sum_{i=1}^r\tl{\pi_{a_i}\gamma\pi_{a_i}^{-1}}.\]
Adding this up over all cycles gives the desired result.
\end{proof}

At last, we are ready to tackle modularity.

\begin{proposition}
Let $\alpha_1, \alpha_2\in\C[\Ep]$. Then there exists a modular form $E\in S_2(\Gamma_0(DM))^{D-\new}$ such that the $n$\textsuperscript{th} coefficient of $E$ equals $\Intpm{\alpha_1}{T_n\alpha_2}$ for all $n$ coprime to $M$.
\end{proposition}
\begin{proof}

By combining $\beta\circ\eta\circ\kappa$, Eichler-Shimura, and Jacquet-Langlands,  we have an Hecke-equivariant isomorphism
\[\C[\Ep]\simeq\C[C(\GOPH)]\simeq H_1(\XO, \C)\simeq H_1(\XO, \C)^*\simeq S_2^B(\Ord)\oplus\overline{S_2^B(\Ord)}\simeq S_2(DM)^{D-\new}\oplus\overline{S_2(DM)}^{D-\new}.\]
The eigenvalues of $\overline{S_2(DM)}^{D-\new}$ are complex conjugates of the eigenvalues of $S_2(DM)^{D-\new}$, but since this space is fixed under $\text{Gal}(\overline{\Q}/\Q)$, we can pair them up. In particular, by Proposition \ref{prop:classmfdecomp} there exists a decomposition,
\[\C[\Ep]=\oplus_{m\mid M} V_m,\]
where each $V_m$ can be decomposed into eigenspaces corresponding to the eigensystems for newforms on $S_2(\Gamma_0(Dm))$, $m\mid M$.
Each eigenspace of $V_m$ can then be decomposed into a basis of eigenforms for all $T_p$ with $p\nmid DM$ and $W_q$ for $q\mid D$.

Assume that $\alpha_1,\alpha_2\in\C[\Ep]$ are basis elements, corresponding to $V_{m_1}, V_{m_2}$ respectively, as well as to the eigensystems $T_p\alpha_i=a_p\alpha_i$, $W_q\alpha_i=a_q\alpha_i$ for $i=1,2$. By Proposition \ref{prop:classmfdecomp}, if these are distinct eigensystems, there exists a $p\nmid DM$ with $a_p\neq a_p'$. Then
\[a_p\Intpm{\alpha_1}{\alpha_2}=\Intpm{T_p\alpha_1}{\alpha_2}=\Intpm{\alpha_1}{T_p\alpha_2}=a_p'\Intpm{\alpha_1}{\alpha_2},\]
whence $\Intpm{\alpha_1}{\alpha_2}=0$. Therefore the only way for this pairing to be non-zero is if $m_1=m_2$ and $a_p=a_p'$ for all $p$. Assume this, and for simplicity assume that the elements are normalized so that $\Intpm{\alpha_1}{\alpha_2}=1$.

Let $E$ correspond to the modular form with coefficients $a_p$, and let $c_n=\Intpm{\alpha_1}{T_n\alpha_2}$. If $p\nmid DM$, then as above, $c_p=a_p$. If $q\mid D$, then $c_q=a_q$ follows from Remark \ref{rem:JLatkinlehner}. Therefore, by combining Proposition \ref{prop:heckeworks} with Proposition \ref{prop:classicalfcoeffs}, it follows that $c_n=a_n$ for all $n$ coprime to $M$.

The result for general $\alpha_i$ follows from writing each element in terms of the basis, which is orthonormal with respect to the signed intersection number.
\end{proof}

Given a modular form in $S_2(\Gamma_0(DM))^{D-\new}$, we can bump up the level to $DM^2$ to only eliminate all coefficients not coprime to $M$ (see for example Proposition 2.4 of \cite{Rib80}). In particular, Theorem \ref{thm:mainthm} follows immediately from this.

\begin{remark}\label{rem:m2q}
When $D=1$, then we are initially working with the open curve $Y_{O}=\Gamma_{\Ord}\backslash\mathbb{H}$. In this case, Poincar\'{e} duality (via the map $\beta$) instead lands in the cohomology of the closed curve $\XO$, relative to the cusps. Eichler-Shimura gives the isomorphism to $S_2(\Gamma_0(M))\oplus\overline{S_2(\Gamma_0(M))}$, as desired.
\end{remark}

\begin{remark}\label{rem:levelDM}
In Section \ref{sec:atkinlehner} we defined $W_q$ for $q^e\mid\mid M$, and most of the subsequent theory still works with this operator. The difficulty comes in picking up the coefficients $a_q$, since the action of $W_q$ on an oldform in $S_2(\Gamma_0(DM))$ does not pick up $a_q$ (let alone $a_{q^n}$). For example, for a newform in $S_2(\Gamma_0(D))$, $W_q$ would need to act like the Hecke operator $T_q$ acting on this space, which does not seem viable.

One alternate way to treat this would be to also consider the superorders $\Ord'\supseteq\Ord$, and form a linear combination of the intersection series for all such superorders. This would allow access to the Hecke operators acting on $S_2(\Gamma_0(Dm))$ for all $m\mid M$, and may allow us to pick up all coefficients.
\end{remark}

\section{Examples}\label{sec:examples}

Algorithms to compute intersection numbers and the action of $T_n$ on optimal embeddings were implemented in PARI/GP, \cite{PARI}. Using these algorithms, we produce a few examples that demonstrate that the modular form corresponding to $\IS_{\phi_1,\phi_2}^{\pm}$ can be non-trivial, does not need to be an eigenform, and does not need to be $M-$new.

The labels of newforms correspond to the labels given in LMFDB (\cite{lmfdb}). The code to generate these examples can be found in the file ``intersectionseries.gp'' in the package \cite{Qquad}.

For a first example, we consider a situation where we get a combination of newforms, so the resulting form is not an eigenform.
\begin{example}
Let $B=\left(\frac{7,5}{\Q}\right)$ be ramified at $5$ and $7$, and let $\Ord$ be the maximal order spanned by $\left\{1,i,\frac{1+j}{2},\frac{i+k}{2}\right\}$. Thus $D=35$, $M=1$, and the dimension of weight two newforms on $\Gamma_0(35)$ is $3$. Label the forms $f,g,\overline{g}$, where $f$ is given by \textsf{35.2.a.a} in LMFDB, and $g$ is given by \textsf{35.2.a.b}. The coefficients of $g$ are given in terms of $\beta=\frac{1+\sqrt{17}}{2}$, and the first few coefficients of $f,g$ are given by
\[f(\tau)=q^1+q^{3}- 2q^{4} - q^{5} + q^{7} - 2q^{9}+O(q^{11}),\]
and
\[g(\tau)= q^1 -\beta q^{2} + ( -1 + \beta ) q^{3} + ( 2 + \beta ) q^{4} + q^{5} -4 q^{6} - q^{7} + ( -4 - \beta ) q^{8} + ( 2 - \beta ) q^{9} -\beta q^{10}+O(q^{11}).\]

Take the optimal embeddings of discriminants $5,12$ given by
\[\phi_1\left(\frac{1+\sqrt{5}}{2}\right)=\frac{1-j}{2},\qquad \phi_2\left(\frac{\sqrt{12}}{2}\right)=\frac{-i-8j+3k}{2}.\]
We compute 
\[\IS_{\phi_1,\phi_2}^{\pm}=q^{2}-q^{3}-q^{4}+q^{8}+q^{9}+q^{10}+O(q^{11}).\]
By matching the coefficients, we have
\[\IS_{\phi_1,\phi_2}^{\pm}=\dfrac{-g(\tau)+\overline{g}(\tau)}{\sqrt{17}}.\]

Next, take the optimal embedding of discriminant $173$ given by 
\[\phi_3\left(\frac{1+\sqrt{173}}{2}\right)=\frac{1-2i+27j+10k}{2}.\]
We compute
\[\IS_{\phi_2,\phi_3}^{\pm}=2q^{1}-q^{2}+3q^{4}+q^{5}-6q^{6}-q^{7}-7q^{8}+q^{9}-q^{10}+O(q^{11}),\]
whence
\[\IS_{\phi_2,\phi_3}^{\pm}=\dfrac{1}{2}f(\tau)+\dfrac{51+\sqrt{17}}{68}g(\tau)+\dfrac{51-\sqrt{17}}{68}\overline{g}(\tau).\]
\end{example}

Next, consider a non-maximal Eichler order.

\begin{example}\label{ex:modform2}
Let $B=\left(\frac{7,-1}{\Q}\right)$ be ramified at $2,7$, and let $\Ord$ be the Eichler order of level $3$ spanned by $\left\{1,i,3j,\frac{1+i+j+k}{2}\right\}$. Thus $D=14$, $M=3$, and the dimensions of the space of weight two newforms on each of $\Gamma_0(14)$ and $\Gamma_0(42)$ is $1$. Let the eigenforms be $f,g$ respectively, so that $f$ is given by the label \textsf{14.2.a.a} in LMFDB, and $g$ is \textsf{42.2.a.a}. The first few terms are given by
\[f(\tau)=q^{1}-q^{2}-2q^{3}+q^{4}+2q^{6}+q^{7}-q^{8}+q^{9}-2q^{12}-4q^{13}-q^{14}+q^{16}+O(q^{17}),\]
and
\[g(\tau)=q^{1}+q^{2}-q^{3}+q^{4}-2q^{5}-q^{6}-q^{7}+q^{8}+q^{9}-2q^{10}-4q^{11}-q^{12}+6q^{13}-q^{14}+2q^{15}+q^{16}+O(q^{17}).\]

Take the embeddings of discriminants $13,24$ given by
\[\phi_1\left(\frac{1+\sqrt{13}}{2}\right)=\frac{1+i+j+k}{2},\qquad \phi_2\left(\frac{\sqrt{24}}{2}\right)=-j-k.\]
We compute 
\[\IS_{\phi_1,\phi_2}^{\pm}=-q^{1}+q^{2}-q^{4}-q^{7}+q^{8}+4q^{13}+q^{14}-q^{16}+O(q^{17}).\]
By matching the coefficients of $q^1,q^2$, we have
\[\IS_{\phi_1,\phi_2}^{\pm}=-f(\tau)+R(q^3),\]
for some power series $R$. An equality of modular forms can be achieved by bumping up the level to access $f(9q)$, and using this to erase all coefficients of $q^{3n}$:
\[\IS_{\phi_1,\phi_2}^{\pm}=-f(\tau)-2f(3\tau)-3f(9\tau).\]

Finally, we demonstrate an example where the old and newforms are non-trivially combined. Let $D_3=45$, let $\phi_3\left(\frac{1+\sqrt{45}}{2}\right)=\frac{1+3i-5j+k}{2}$, and
\[\IS_{\phi_2,\phi_3}^{\pm}=q^{1}+q^{4}-q^{5}-q^{10}-2q^{11}+q^{13}-q^{14}+q^{16}+O(q^{17}).\]
Matching coefficients gives
\[\IS_{\phi_2,\phi_3}^{\pm}=\dfrac{f(\tau)+g(\tau)}{2}+R(q^3),\]
for some power series $R$.
\end{example}

\bibliographystyle{alpha}
\bibliography{../references}

\begin{thebibliography}{{LMF}21}

\bibitem[AL70]{AL70}
A.~O.~L. Atkin and J.~Lehner.
\newblock Hecke operators on {$\Gamma _{0}(m)$}.
\newblock {\em Math. Ann.}, 185:134--160, 1970.

\bibitem[BD96]{BD96}
M.~Bertolini and H.~Darmon.
\newblock Heegner points on {M}umford-{T}ate curves.
\newblock {\em Invent. Math.}, 126(3):413--456, 1996.

\bibitem[Cha20]{SC20}
Sara Chari.
\newblock Metacommutation of primes in central simple algebras.
\newblock {\em J. Number Theory}, 206:296--309, 2020.

\bibitem[DV13]{DemV13}
Lassina Demb\'{e}l\'{e} and John Voight.
\newblock Explicit methods for {H}ilbert modular forms.
\newblock In {\em Elliptic curves, {H}ilbert modular forms and {G}alois
  deformations}, Adv. Courses Math. CRM Barcelona, pages 135--198.
  Birkh\"{a}user/Springer, Basel, 2013.

\bibitem[DV21]{DV20}
Henri Darmon and Jan Vonk.
\newblock Singular moduli for real quadratic fields: a rigid analytic approach.
\newblock {\em Duke Math. J.}, 170(1):23--93, 2021.

\bibitem[GKZ87]{GKZ87}
B.~Gross, W.~Kohnen, and D.~Zagier.
\newblock Heegner points and derivatives of {$L$}-series. {II}.
\newblock {\em Math. Ann.}, 278(1-4):497--562, 1987.

\bibitem[GV11]{GV11}
Matthew Greenberg and John Voight.
\newblock Computing systems of {H}ecke eigenvalues associated to {H}ilbert
  modular forms.
\newblock {\em Math. Comp.}, 80(274):1071--1092, 2011.

\bibitem[GZ85]{GZ85}
Benedict~H. Gross and Don~B. Zagier.
\newblock On singular moduli.
\newblock {\em J. Reine Angew. Math.}, 355:191--220, 1985.

\bibitem[{LMF}21]{lmfdb}
The {LMFDB Collaboration}.
\newblock The {L}-functions and modular forms database.
\newblock \url{http://www.lmfdb.org}, 2021.
\newblock [Online].

\bibitem[PAR21]{PARI}
The PARI~Group, Univ. Bordeaux.
\newblock {\em PARI/GP version \texttt{2.13.2}}, 2021.
\newblock available from \url{http://pari.math.u-bordeaux.fr/}.

\bibitem[Rib80]{Rib80}
Kenneth~A. Ribet.
\newblock Twists of modular forms and endomorphisms of abelian varieties.
\newblock {\em Math. Ann.}, 253(1):43--62, 1980.

\bibitem[Ric21a]{JR21shim}
James Rickards.
\newblock Counting intersection numbers of closed geodesics on {S}himura
  curves.
\newblock \url{https://arxiv.org/abs/2104.01968}, 2021.

\bibitem[Ric21b]{Qquad}
James Rickards.
\newblock Q- {Quadratic}.
\newblock \url{https://github.com/JamesRickards-Canada/Q-Quadratic}, 2021.

\bibitem[Voi09]{JV09}
John Voight.
\newblock Computing fundamental domains for {F}uchsian groups.
\newblock {\em J. Th\'{e}or. Nombres Bordeaux}, 21(2):469--491, 2009.

\bibitem[Voi21]{JV21}
John Voight.
\newblock {\em Quaternion algebras}, volume 288 of {\em Graduate Texts in
  Mathematics}.
\newblock Springer, Cham, [2021] \copyright 2021.

\bibitem[Zag75]{Zag75}
Don Zagier.
\newblock Nombres de classes et formes modulaires de poids {$3/2$}.
\newblock {\em C. R. Acad. Sci. Paris S\'{e}r. A-B}, 281(21):Ai, A883--A886,
  1975.

\end{thebibliography}
\end{document}